\documentclass{amsart}
\usepackage{graphicx}
\usepackage{amssymb}
\usepackage{wrapfig}
\newtheorem{thm}{Theorem}[section]
\newtheorem{cor}[thm]{Corollary}
\newtheorem{lem}[thm]{Lemma}
\newtheorem{prop}[thm]{Proposition}
\theoremstyle{definition}
\newtheorem{defn}[thm]{Definition}
\theoremstyle{remark}

\numberwithin{equation}{section}
\newcommand{\norm}[1]{\left\Vert#1\right\Vert}
\newcommand{\abs}[1]{\left\vert#1\right\vert}
\newcommand{\set}[1]{\left\{#1\right\}}

\newcommand{\eps}{\varepsilon}

\newcommand{\dbar}{\bar\partial}
\newcommand{\ddbar}{\partial\bar\partial}

\DeclareMathOperator{\re}{Re}
\DeclareMathOperator{\im}{Im}
\DeclareMathOperator{\supp}{supp}
\DeclareMathOperator{\dist}{dist}

\begin{document}

\title[On Competing Definitions for the Diederich-Forn{\ae}ss Index]{On Competing Definitions for the Diederich-Forn{\ae}ss Index}%
\author{Phillip S. Harrington}%
\address{SCEN 309, 1 University of Arkansas, Fayetteville, AR 72701}%
\email{psharrin@uark.edu}%


\begin{abstract}
Let $\Omega\subset\mathbb{C}^n$ be a bounded pseudoconvex domain.  We define the Diederich-Forn{\ae}ss index with respect to a family of functions to be the supremum over the set of all exponents $0<\eta<1$ such that there exists a function $\rho_\eta$ in this family that is comparable to $-\dist(z,b\Omega)$ on $\Omega$ and such that $-(-\rho_\eta)^\eta$ is plurisubharmonic on $\Omega$.  We first prove that computing the Diederich-Forn{\ae}ss index with respect to the family of upper semi-continuous functions is the same as computing the Diederich-Forn{\ae}ss index with respect to the family of Lipschitz functions.  When the boundary of $\Omega$ is $C^k$, $k\geq 2$, we prove that the Diederich-Forn{\ae}ss index with respect to the family of $C^k$ functions is the same as the Diederich-Forn{\ae}ss index with respect to the family of $C^2$ functions.
\end{abstract}
\maketitle

\tableofcontents


\section{Introduction}

A domain $\Omega\subset\mathbb{C}^n$ is said to be pseudoconvex if it admits a smooth, bounded, strictly plurisubharmonic exhaustion function.  Let $\delta_\Omega(z)=\dist(z,b\Omega)$.  Oka's Lemma states that $\Omega$ is pseudoconvex if and only if $-\log\delta_\Omega$ is plurisubharmonic on $\Omega$.  For many applications, it is desirable to obtain a bounded plurisubharmonic exhaustion function.  When $\Omega$ is bounded and $b\Omega$ is $C^2$, Diederich and Forn{\ae}ss proved that there exists an exponent $0<\eta<1$ and a $C^2$ defining function $\rho_\eta$ for $\Omega$ with the property that $-(-\rho_\eta)^\eta$ is strictly plurisubharmonic on $\Omega$, thus providing a bounded plurisubharmonic exhaustion function for all such domains.  The supremum of all such exponents $\eta$ for a given domain $\Omega$ has come to be known as the Diederich-Forn{\ae}ss index for $\Omega$, although in the present paper we will identify this as the strong Diederich-Forn{\ae}ss index.

In contrast, we will also consider the weakest notion of the Diederich-Forn{\ae}ss index that makes sense:
\begin{defn}
  Let $\Omega\subset\mathbb{C}^n$ be a bounded pseudoconvex domain.  We define the weak Diederich-Forn{\ae}ss index $DF_w(\Omega)$ to be the supremum over all exponents $0<\eta<1$ such that there exists an upper semi-continuous plurisubharmonic function $\lambda_\eta$ on $\Omega$ satisfying
  \[
    -c(\delta_\Omega)^\eta\leq\lambda_\eta\leq-(\delta_\Omega)^\eta
  \]
  on $\Omega$ for some $c>1$.  If no such function exists, then we say that $DF_w(\Omega)=0$.
\end{defn}
We could also include an arbitrary constant in the upper bound on $\lambda_\eta$, but we may always re-scale $\lambda_\eta$ so that this constant is equal to one.  By a result of Richberg \cite{Ric68}, we may assume that $\lambda_\eta\in C^\infty(\Omega)$, but we wish to examine the regularity properties of $\lambda_\eta$ on $\overline\Omega$ more closely.

The author has shown that the weak Diederich-Forn{\ae}ss index is positive when $b\Omega$ is Lipschitz in \cite{Har08a}.  This was motivated by work of Berndtsson and Charpentier \cite{BeCh00}, in which they show that the Bergman Projection is regular in the Sobolev space $W^s(\Omega)$ whenever $\Omega\subset\mathbb{C}^n$ is a bounded Lipschitz domain and $0\leq s<\frac{1}{2}DF_w(\Omega)$.

The original result of Diederich and Forn{\ae}ss is, \textit{a priori}, much stronger.  To clarify this, we introduce the following family of definitions:
\begin{defn}
  Let $\Omega\subset\mathbb{C}^n$ be a bounded pseudoconvex domain.  Let $\mathcal{F}(\overline\Omega)$ be a family of functions on $\overline\Omega$ that are at least upper semi-continuous.  We define $DF_{\mathcal{F}}(\Omega)$ to be the supremum over all exponents $0<\eta<1$ such that there exists a function $\rho_\eta\in\mathcal{F}(\overline\Omega)$ satisfying
  \[
    -c\delta_\Omega\leq\rho_\eta\leq-\delta_\Omega
  \]
  on $\overline\Omega$ for some $c>1$ and such that $-(-\rho_\eta)^\eta$ is plurisubharmonic on $\Omega$.  If no such function exists, then we say that $DF_{\mathcal{F}}(\Omega)=0$.
\end{defn}
We clearly have $DF_{\mathcal{F}}(\Omega)\leq DF_w(\Omega)$.  If $b\Omega$ is $C^1$ and $\mathcal{F}(\overline\Omega)\subset C^1(\overline\Omega)$, then $\rho_\eta$ necessarily extends to a defining function for $\Omega$.  If $\mathcal{F}=\Lambda^1(\overline\Omega)$, the space of Lipschitz functions on $\Omega$, then it is not necessarily true that $\nabla\rho_\eta$ is uniformly bounded away from zero almost everywhere, and hence $\rho_\eta$ may not be a Lipschitz defining function.  For example, if $\rho$ is a smooth defining function for a smooth domain, then $\tilde\rho=\rho+C\rho^2\sin(\rho^{-1})$ is a Lipschitz function on $\Omega$ that is comparable to $-\delta_\Omega$ near $b\Omega$, but for $C$ sufficiently large the gradient of $\tilde\rho$ will not be uniformly bounded away from zero on any neighborhood of $b\Omega$.

With this terminology, we can easily restate many results on the Diederich-Forn{\ae}ss index.  Diederich and Forn{\ae}ss originally proved that $DF_{C^2}(\Omega)>0$ whenever $\Omega$ has $C^2$ boundary.  In \cite{Ran81}, Range simplified their proof to show that $DF_{C^3}(\Omega)>0$ whenever $\Omega$ has $C^3$ boundary.  In \cite{HeFo07} and \cite{HeFo08}, Herbig and Forn{\ae}ss prove that $DF_{C^\infty}(\Omega)=1$ whenever $\Omega$ has a smooth boundary and admits a smooth defining function $\rho$ that is plurisubharmonic on the boundary.  Krantz, Liu, and Peloso \cite{KLP18} looked at more general sufficient conditions for $DF_{C^\infty}(\Omega)=1$ when $\Omega\subset\mathbb{C}^2$, and Liu generalized this study to $\mathbb{C}^n$ in \cite{Liu19}.  Kohn \cite{Koh99} and Pinton and Zampieri \cite{PiZa14} have used $DF_{C^\infty}(\Omega)$ to prove regularity for the Bergman projection and the $\bar\partial$-Neumann operator in Sobolev spaces, and the author \cite{Har11} and Liu \cite{Liu19b} have carried out similar work in the special case when $DF_{C^\infty}(\Omega)=1$.  Abdulsahib and the author \cite{AbHa19} have estimated $DF_{C^\infty}(\Omega)$ on Hartogs domains with smooth boundaries.

Our main goal for this paper is to clarify the relationship between these various indices as follows:
\begin{thm}
  Let $\Omega\subset\mathbb{C}^n$ be a bounded, pseudoconvex domain.  Then $DF_{\Lambda^1}(\Omega)=DF_w(\Omega)$.  If $b\Omega$ is $C^k$ for some $k\geq 2$, then $DF_{C^k}(\Omega)=DF_{C^2}(\Omega)$.
\end{thm}
This will follow from Corollaries \ref{cor:weak_DF_index} and \ref{cor:strong_DF_Index} below.  Motivated by this, we define
\begin{defn}
  Let $\Omega\subset\mathbb{C}^n$ be a bounded pseudoconvex domain with $C^2$ boundary.  We define the strong Diederich-Forn{\ae}ss index $DF_s(\Omega)$ to be $DF_{C^2}(\Omega)$.
\end{defn}

This leaves open the question of whether $DF_w(\Omega)=DF_s(\Omega)$ for bounded pseudoconvex domains with $C^2$ boundaries.  \textit{A priori}, we observe that the strong Diederich-Forn{\ae}ss index seems to involve deeper geometric properties of the boundary.  This was explored in depth by Liu in \cite{Liu17a}.  As in Liu's work, we will show in Proposition \ref{prop:Ck_equivalence} that on domains with $C^3$ boundaries the strong Diederich-Forn{\ae}ss index may be completely characterized by the existence of a family of functions on the boundary of $\Omega$ satisfying a differential inequality.  In Proposition \ref{prop:C2_equivalence}, we will prove an analogous result on domains with $C^2$ boundaries, except that this will require working on an internal neighborhood of the boundary rather than the boundary itself.  In Section \ref{sec:estimates}, we will use this characterization to estimate the strong Diederich-Forn{\ae}ss index on domains satisfying several key hypotheses.

In \cite{Liu17a}, Liu has explicitly computed $DF_s(\Omega)$ when $\Omega$ is the worm domain of Diederich and Forn{\ae}ss \cite{DiFo77a}.  If one could show that $DF_w(\Omega)>DF_s(\Omega)$ for the worm domain, then Berndtsson and Charpentier's work \cite{BeCh00} could be used to improve the known Sobolev regularity for the Bergman Projection on the worm domain.  Since Barrett \cite{Bar92} has already shown that there is an upper bound on the Sobolev regularity for the Bergman Projection on the worm domain, this would be significant.

We note that Diederich and Forn{\ae}ss actually proved that $-(-\rho_\eta)^\eta$ was strictly plurisubharmonic, while we have defined our indices with respect to plurisubharmonic functions.  We will see in our proofs that these definitions are equivalent.

The author would like to thank Peter Pflug for bringing to his attention the possible discrepancy between $DF_{C^2}(\Omega)$ and $DF_{C^\infty}(\Omega)$.

\section{Notation and Preliminary Computations}

For a domain $\Omega\subset\mathbb{R}^n$, let $\delta_\Omega(x)$ denote the distance from $x\in\mathbb{R}^n$ to $b\Omega$.  When a unique point $y\in b\Omega$ exists satisfying $|x-y|=\delta_\Omega(x)$, we write $y=\xi_\Omega(x)$.  The signed distance function is defined by $\tilde\delta_\Omega(x)=\delta_\Omega(x)$ outside $\Omega$ and $\tilde\delta_\Omega(x)=-\delta_\Omega(x)$ on $\Omega$.  For $x\notin b\Omega$, Federer has shown that $\xi_\Omega(x)$ is defined if and only if $\delta_\Omega$ is differentiable at $x$, and these quantities satisfy the relationship $\xi_\Omega(x)=x-\delta_\Omega(x)\nabla\delta_\Omega(x)$ by Theorem 4.8 (3) in \cite{Fed59}.  Equivalently,
\begin{equation}
\label{eq:projection}
  \xi_\Omega(x)=x-\tilde\delta_\Omega(x)\nabla\tilde\delta_\Omega(x),
\end{equation}
for $x\notin b\Omega$ such that $\delta_\Omega$ is differentiable at $x$.

In $\mathbb{C}^n$, we let $\omega(z)=\frac{i}{2}\ddbar|z|^2$ denote the K\"ahler form for the Euclidean metric.  As a notational convenience, if $\Theta$ is a real $(1,1)$-form, we define the action of $\Theta$ on $T^{1,0}\times T^{0,1}$ by $\Theta(X,\bar Y)=\Theta(-i X\wedge\bar Y)$ for any $X,Y\in T^{1,0}$.  Hence, for example, $\omega(X,\bar Y)=\left<X,Y\right>$, where $\left<\cdot,\cdot\right>$ denotes the Hermitian inner product on $T^{1,0}\times T^{1,0}$.

The following lemma is elementary, but we record it separately since we will need the precise values of the constants several times in this section:
\begin{lem}
\label{lem:basic_inequality}
  Let $a,b\in\mathbb{C}$ and let $0<\epsilon<1$.  Then
  \begin{equation}
  \label{eq:basic_inequality}
    |a+b|^2\geq \epsilon|a|^2-\frac{\epsilon}{1-\epsilon}|b|^2.
  \end{equation}
\end{lem}

\begin{proof}
  We have
  \[
    0\leq\frac{1}{1-\epsilon}|(1-\epsilon) a+b|^2=(1-\epsilon)|a|^2+2\re(a\bar b)+\frac{1}{1-\epsilon}|b|^2=|a+b|^2-\epsilon|a|^2+\frac{\epsilon}{1-\epsilon}|b|^2.
  \]
\end{proof}

We will frequently need to patch together plurisubharmonic functions without sacrificing regularity, so the following lemma will be helpful.
\begin{lem}
\label{lem:psi_xi_construction}
  For every $\xi>0$, there exists a smooth, convex function $\psi_\xi(x,y)$ on $\mathbb{R}^2$ with the property that $\psi_\xi(x,y)=x$ when $x\geq y+\xi$ and $\psi_\xi(x,y)=y$ when $y\geq x+\xi$.  Such a function necessarily satisfies $\frac{\partial}{\partial x}\psi_\xi(x,y)+\frac{\partial}{\partial y}\psi_\xi(x,y)=1$ on $\mathbb{R}^2$.
\end{lem}

\begin{proof}
  Let $\chi$ be a smooth, positive, radially symmetric function on $\mathbb{R}^2$ supported in $\overline{B(0,1)}$ such that $\int\chi=1$, and set $\chi_\xi(x,y)=\frac{2}{\xi^2}\chi\left(\frac{\sqrt{2}x}{\xi},\frac{\sqrt{2}y}{\xi}\right)$, so that $\int\chi_\xi=1$ and $\chi_\xi$ is supported in $\overline{B\left(0,\frac{\xi}{\sqrt{2}}\right)}$.  Let $\psi_0(x,y)=\max\{x,y\}$, and set $\psi_\xi(x,y)=\psi_0\ast\chi_\xi$.  Since $\psi_0$ is convex, $\psi_\xi$ will also be convex.  When $x\geq y+\xi$, $\psi_0(a,b)=a$ on $B\left((x,y),\frac{\xi}{\sqrt{2}}\right)$, so $(\psi_0\ast\chi_\xi)(x,y)=x$ (convolution against a normalized radially symmetric function will preserve harmonic functions, and linear functions are harmonic).  Similarly, when $y\geq x+\xi$, $\psi_\xi(x,y)=y$.

  Fix $(x_0,y_0)\in\mathbb{R}^2$.  Then convexity of $\psi_\xi$ implies that $f(x,y)=\psi_\xi(x_0,y_0)+\frac{\partial}{\partial x}\psi_\xi(x_0,y_0)(x-x_0)+\frac{\partial}{\partial y}\psi_\xi(x_0,y_0)(y-y_0)$ satisfies $f(x,y)\leq\psi_\xi(x,y)$ for all $(x,y)\in\mathbb{R}^2$.  In particular, $y+\xi-f(y+\xi,y)\geq 0$ for all $y\in\mathbb{R}$.  Now
  \[
    \lim_{y\rightarrow\infty}\frac{y+\xi-f(y+\xi,y)}{y}=1-\frac{\partial}{\partial x}\psi_\xi(x_0,y_0)-\frac{\partial}{\partial y}\psi_\xi(x_0,y_0),
  \]
  so
  \[
    \frac{\partial}{\partial x}\psi_\xi(x_0,y_0)+\frac{\partial}{\partial y}\psi_\xi(x_0,y_0)\leq 1.
  \]
  Similarly,
  \[
    \lim_{y\rightarrow-\infty}\frac{y+\xi-f(y+\xi,y)}{-y}=-1+\frac{\partial}{\partial x}\psi_\xi(x_0,y_0)+\frac{\partial}{\partial y}\psi_\xi(x_0,y_0),
  \]
  so
  \[
    \frac{\partial}{\partial x}\psi_\xi(x_0,y_0)+\frac{\partial}{\partial y}\psi_\xi(x_0,y_0)\geq 1.
  \]
\end{proof}

As a useful consequence, we have the following:
\begin{lem}
\label{lem:psh_extensions}
  Let $\Omega\subset\mathbb{C}^n$ be a bounded pseudoconvex domain, let $U_\eta$ be an open neighborhood of $b\Omega$, and let $\lambda_\eta$ be an upper semi-continuous plurisubharmonic function satisfying $-c(\delta_\Omega)^\eta\leq\lambda_\eta\leq-(\delta_\Omega)^\eta$ on $U_\eta\cap\Omega$ for some $c>1$ and $0<\eta<1$.  Then there exists an upper semi-continuous plurisubharmonic function $\tilde\lambda_\eta$ on $\Omega$ such that $\tilde\lambda_\eta=\lambda_\eta$ on $\tilde U_\eta\cap\Omega$ for some neighborhood $\tilde U_\eta$ of $b\Omega$.  If $\lambda_\eta$ is $C^k$ for some $k\geq 1$ (resp. Lipschitz) on $U_\eta\cap\Omega$, then $\tilde\lambda_\eta$ also $C^k$ (resp. Lipschitz) on $\Omega$.  If $\lambda_\eta$ is strictly plurisubharmonic on $U_\eta\cap\Omega$, then $\tilde\lambda_\eta$ is also strictly plurisubharmonic on $\Omega$.
\end{lem}

\begin{proof}
  Let $\delta_0>0$ satisfy $z\in U_\eta$ whenever $\delta_\Omega(z)\leq\delta_0$.  Let $r=\sup_\Omega|z|$.  Fix $\xi<\frac{1}{3}\delta_0^\eta$ and let $\psi_\xi$ be given by Lemma \ref{lem:psi_xi_construction}.  On $\Omega\backslash U_\eta$, let $\tilde\lambda_\eta=\left(\frac{|z|^2}{3r^2}-\frac{2}{3}\right)\delta_0^\eta$, and on $U$, let $\tilde\lambda_\eta=\psi_\xi\left(\lambda_\eta,\left(\frac{|z|^2}{3r^2}-\frac{2}{3}\right)\delta_0^\eta\right)$.
\end{proof}

Hence, in all that follows we will be able to restrict to a neighborhood of $b\Omega$ when computing the Diederich-Forn{\ae}ss Indices.

\section{The Weak Diederich-Forn{\ae}ss Index}

Our primary goal for this section is to show that the result of Richberg \cite{Ric68} can be strengthened in the context of the Diederich-Forn{\ae}ss index:
\begin{prop}
\label{prop:weak_DF_Index}
  Let $\Omega\subset\mathbb{C}^n$ be a bounded domain and let $\lambda_\eta$ be an upper semi-continuous plurisubharmonic function on $\Omega$ such that $-C_\eta\delta_\Omega^\eta\leq\lambda_\eta\leq -\delta_\Omega^\eta$ on $\Omega$ for some $C_\eta>1$.  Then, for every $0<s<\eta$, there exist constants $C_s>1$, $E_s>0$, and a function $\rho_s\in C^\infty(\Omega)$ that is Lipschitz on $\overline\Omega$ and satisfies
  \[
    -C_s^{1/s}\delta_\Omega\leq \rho_s\leq -\delta_\Omega
  \]
  and
  \[
    i\ddbar(-(-\rho_s)^s)\geq E_s(\delta_\Omega)^s\omega.
  \]
\end{prop}

\begin{proof}
  Let $\chi\in C^\infty_0(\mathbb{C}^n)$ be a radially symmetric nonnegative function such that $\supp\chi=\overline{B(0,1)}$ and $\int_{\mathbb{C}^n}\chi dV=1$.  For $\eps>0$, set $\chi_\eps(z)=\frac{1}{\eps^{2n}}\chi\left(\frac{z}{\eps}\right)$.  Then $\supp\chi_\eps=\overline{B(0,\eps)}$ and $\int_{\mathbb{C}^n}\chi_\eps dV=1$.

  Let $\Omega_\eps=\{z\in\Omega:\delta_\Omega(z)>\eps\}$.  On $\Omega_\eps$, we may define $\lambda_\eps=\lambda\ast\chi_\eps$.  Since plurisubharmonicity is preserved by convolution with $\chi_\eps$, $\lambda_\eps$ is plurisubharmonic on $\Omega_\eps$.  By the sub-mean-value property, $\lambda(z)\leq\lambda_\eps(z)$ for $z\in\Omega_\eps$.  On the other hand, for $z\in\Omega_\eps$ we have
  \[
    \lambda_\eps(z)\leq\sup_{w\in B(z,\eps)}\lambda(w)\leq-\inf_{w\in B(z,\eps)}(\delta_\Omega(w))^\eta\leq -(\delta_\Omega(z)-\eps)^\eta,
  \]
  so
  \begin{equation}
  \label{eq:lambda_eps_estimate}
    -C_\eta\delta_\Omega^\eta\leq\lambda_\eps\leq-(\delta_\Omega-\eps)^\eta
  \end{equation}
  on $\Omega_\eps$.  Since $\nabla\lambda_\eps=\lambda\ast\nabla\chi_\eps$ and $\norm{\nabla\chi_\eps}_{L^1(\mathbb{C}^n)}\leq O\left(\frac{1}{\eps}\right)$, we have
  \begin{equation}
  \label{eq:lambda_eps_gradient_estimate}
    \abs{\nabla\lambda_\eps}\leq G(\delta_\Omega+\eps)^\eta\eps^{-1}
  \end{equation}
  on $\Omega_\eps$ for some constant $G>0$ independent of $\eps$.

  Fix $a>1$.  Since $\eta>s$, we have
  \[
    \lim_{b\rightarrow\infty}\frac{(b-1)^\eta b^{-s/2}+(a-1)^{\eta}a^{-s/2}}{(ab)^{\eta/2}(a^{-s/2}+b^{-s/2})}=\infty,
  \]
  so we may fix $b>a$ satisfying
  \begin{equation}
  \label{eq:b_characterization}
    \frac{(b-1)^\eta b^{-s/2}+(a-1)^\eta a^{-s/2}}{(ab)^{\eta/2}(a^{-s/2}+b^{-s/2})}>C_\eta.
  \end{equation}
  Set
  \[
    A=\eta^{-1} s^{s/\eta}\frac{b^\eta-a^\eta}{(b-1)^\eta-(a-1)^\eta}\left(\frac{(\eta-s)(a^{-s}-b^{-s})}{b^{\eta-s}-a^{\eta-s}}\right)^{(\eta-s)/\eta}
  \]
  and
  \[
    B=A\frac{(b-1)^\eta b^{-s}-(a-1)^\eta a^{-s}}{a^{-s}-b^{-s}}.
  \]
  The function $y=(x-1)^\eta x^{-s}$ is increasing with respect to $x$ when $x>1$, so $A>0$ and $B>0$.  Let $r=\sup_\Omega|z|$.  By \eqref{eq:b_characterization}, we may choose $E_s$ to satisfy
  \begin{equation}
  \label{eq:E_s_characterization}
     0<E_s<2A\left(\frac{(b-1)^\eta b^{-s/2}+(a-1)^\eta a^{-s/2}}{(ab)^{\eta/2}(a^{-s/2}+b^{-s/2})}-C_\eta\right)(ab)^{\eta/2}b^{-s} r^{-2}.
  \end{equation}
  For $z\in\Omega_\eps$, we set
  \[
    \sigma_\eps(z)=A\eps^{s-\eta}\lambda_\eps(z)-B\eps^s+\frac{1}{2}E_s b^{s}\eps^s(|z|^2-r^2).
  \]

  Let
  \[
    K_\eps=\{z\in\Omega_\eps:\eps a\leq\delta_\Omega(z)\leq\eps b\}.
  \]
  Since $\lambda_\eps$ is plurisubharmonic on $\Omega_\eps$, on $K_\eps$ we have
  \begin{equation}
  \label{eq:sigma_eps_hessian_lower_bound}
    i\ddbar\sigma_\eps\geq E_s(\delta_\Omega)^s\omega.
  \end{equation}
  From \eqref{eq:lambda_eps_gradient_estimate}, we obtain
  \[
    \abs{\nabla\sigma_\eps}\leq A\eps^{s-1}G\left(\frac{\delta_\Omega}{\eps}+1\right)^\eta+E_s b^s\eps^s r,
  \]
  so on $K_\eps$ we have
  \begin{equation}
  \label{eq:sigma_eps_gradient_bound}
    \abs{\nabla\sigma_\eps}\leq Ab^{1-s}G(b+1)^\eta(\delta_\Omega)^{s-1}+E_s b^s a^{-s} r(\delta_\Omega)^s.
  \end{equation}

  Using \eqref{eq:lambda_eps_estimate}, we have
  \[
    \sigma_\eps\leq -A\eps^s\left(\frac{\delta_\Omega}{\eps}-1\right)^\eta-B\eps^s.
  \]
  The function $y=(x^{1/\eta}-1)^\eta$ is concave down, so it is bounded below by the secant line connecting $(a^\eta,(a-1)^\eta)$ to $(b^\eta,(b-1)^\eta)$.  Hence, for $a^\eta\leq x\leq b^\eta$, we have
  \[
    (x^{1/\eta}-1)^\eta\geq\frac{(b-1)^\eta-(a-1)^\eta}{b^\eta-a^\eta}x+\frac{b^\eta(a-1)^\eta-a^\eta(b-1)^\eta}{b^\eta-a^\eta}.
  \]
  Letting $x=(\delta_\Omega(z)/\eps)^\eta$ for $z\in K_\eps$, we find that on $K_\eps$ we have
  \[
    \sigma_\eps\leq-A\eps^s\frac{(b-1)^\eta-(a-1)^\eta}{b^\eta-a^\eta}\left(\frac{\delta_\Omega}{\eps}\right)^\eta-A\eps^s\frac{b^\eta(a-1)^\eta-a^\eta(b-1)^\eta}{b^\eta-a^\eta}-B\eps^s
  \]
  Note that $A\frac{b^\eta(a-1)^\eta-a^\eta(b-1)^\eta}{b^\eta-a^\eta}+B=A\frac{((b-1)^\eta-(a-1)^\eta) (b^{\eta-s}-a^{\eta-s})}{(b^\eta-a^\eta)(a^{-s}-b^{-s})}$, so on $K_\eps$
  \begin{equation}
  \label{eq:sigma_eps_upper_bound}
    \sigma_\eps\leq-A\eps^s\left(\frac{(b-1)^\eta-(a-1)^\eta}{b^\eta-a^\eta}\right)\left(\left(\frac{\delta_\Omega}{\eps}\right)^\eta+\frac{b^{\eta-s}-a^{\eta-s}}{a^{-s}-b^{-s}}\right).
  \end{equation}
  For $z\in b K_\eps$, either $\delta_\Omega(z)=a\eps$ or $\delta_\Omega(z)=b\eps$, so \eqref{eq:sigma_eps_upper_bound} implies
  \begin{equation}
  \label{eq:sigma_eps_critical_upper_bound}
    \sigma_\eps\leq-A(\delta_\Omega)^s\left(\frac{(b-1)^\eta-(a-1)^\eta}{b^s-a^s}\right)\text{ on }b K_\eps.
  \end{equation}
  Using elementary calculus, one can check that the function $y=x^{\eta-s}+\frac{b^{\eta-s}-a^{\eta-s}}{a^{-s}-b^{-s}}x^{-s}$ has a unique minimum (for
  $x>0$) at $x=\left(\frac{s(b^{\eta-s}-a^{\eta-s})}{(\eta-s)(a^{-s}-b^{-s})}\right)^{1/\eta}$.  After some simplification at this critical point, we find that
  \[
    x^{\eta-s}+\frac{b^{\eta-s}-a^{\eta-s}}{a^{-s}-b^{-s}}x^{-s}\geq\eta s^{-s/\eta}\left(\frac{b^{\eta-s}-a^{\eta-s}}{(\eta-s)(a^{-s}-b^{-s})}\right)^{(\eta-s)/\eta}.
  \]
  For $z\in K_\eps$, we may substitute $x=\frac{\delta_\Omega(z)}{\eps}$ in this inequality.  Recalling the definition of $A$, we see that \eqref{eq:sigma_eps_upper_bound} implies
  \begin{equation}
  \label{eq:sigma_eps_normalized_upper_bound}
    \sigma_\eps\leq-(\delta_\Omega)^s\text{ on }K_\eps.
  \end{equation}

  On the other hand, \eqref{eq:lambda_eps_estimate} also gives us
  \[
    \sigma_\eps\geq -A\eps^{s}C_\eta\left(\frac{\delta_\Omega}{\eps}\right)^\eta-B\eps^s-\frac{1}{2}E_s b^{s}\eps^s r^2.
  \]
  When $z\in K_\eps$ satisfies $\delta_\Omega(z)=\sqrt{ab}\eps$, we have
  \[
    \sigma_\eps(z)\geq-(\delta_\Omega(z))^s(ab)^{-s/2}\left( A C_\eta(ab)^{\eta/2}+B+\frac{1}{2}E_s b^{s} r^2\right).
  \]
  If we set
  \[
    \xi=(ab)^{-s/2}\left(A\left(\frac{(b-1)^\eta b^{-s/2}+(a-1)^\eta a^{-s/2}}{(ab)^{\eta/2}(a^{-s/2}+b^{-s/2})}-C_\eta\right)(ab)^{\eta/2}-\frac{1}{2} b^{s} r^2 E_s\right),
  \]
  then \eqref{eq:E_s_characterization} guarantees $\xi>0$.  Hence
  \[
    \sigma_\eps(z)\geq-(\delta_\Omega(z))^s\left((ab)^{-s/2}\left(B+A\frac{(b-1)^\eta b^{-s/2}+(a-1)^\eta a^{-s/2}}{a^{-s/2}+b^{-s/2}}\right)-\xi\right)
  \]
  whenever $z\in K_\eps$ satisfies $\delta_\Omega(z)=\sqrt{ab}\eps$.  Note that
  \[
    B+A \frac{(b-1)^\eta b^{-s/2}+(a-1)^\eta a^{-s/2}}{a^{-s/2}+b^{-s/2}}=A \frac{(b-1)^\eta-(a-1)^\eta}{a^{-s}-b^{-s}}(ab)^{-s/2},
  \]
  so
  \begin{equation}
  \label{eq:sigma_eps_critical_lower_bound}
    \sigma_\eps(z)\geq-(\delta_\Omega(z))^s\left(A \frac{(b-1)^\eta-(a-1)^\eta}{b^s-a^s}-\xi\right)\text{ when }\delta_\Omega(z)=\sqrt{ab}\eps.
  \end{equation}

  Combining \eqref{eq:sigma_eps_critical_upper_bound} and \eqref{eq:sigma_eps_critical_lower_bound}, we see that
  \[
    \xi\leq\inf_{\eps>0}\inf_{\{z\in K_\eps:\delta_\Omega(z)=\sqrt{ab}\eps\}}\frac{\sigma_\eps(z)}{(\delta_\Omega(z))^s}-\sup_{\eps>0}\sup_{z\in b K_\eps}\frac{\sigma_\eps(z)}{(\delta_\Omega(z))^s}.
  \]
  Let $\psi_\xi$ be the function given by Lemma \ref{lem:psi_xi_construction}.  Fix
  \[
    a^{-1}\sup_{\Omega}\delta_\Omega>\eps_0>(ab)^{-1/2}\sup_{\Omega}\delta_\Omega,
  \]
  and for $j\in\mathbb{N}$, let $\eps_j=\eps_0\left(\frac{a}{b}\right)^{j/2}$.  Define $\sigma_s$ by
  \[
    \sigma_s(z)=\eps_j^s\psi_\xi\left(\frac{\sigma_{\eps_j}(z)}{\eps_j^s},\frac{\sigma_{\eps_{j+1}}(z)}{\eps_j^s}\right)\text{ when }a\eps_j\leq\delta_\Omega(z)<\sqrt{ab}\eps_j
  \]
  for all $j\geq 0$.  Note that $a\eps_j=\sqrt{ab}\eps_{j+1}$ and $\sqrt{ab}\eps_j=b\eps_{j+1}$.  When $z\in\Omega$ satisfies $\delta_\Omega(z)=\sqrt{ab}\eps_j$, we have
  \[
    \frac{\sigma_{\eps_j}(z)}{\eps_j^s}=(ab)^{s/2}\frac{\sigma_{\eps_j}(z)}{(\delta_\Omega(z))^s}\geq(ab)^{s/2}\left(\frac{\sigma_{\eps_{j+1}}(z)}{(\delta_\Omega(z))^s}+\xi\right)=\frac{\sigma_{\eps_{j+1}}(z)}{\eps_j^s}+(ab)^{s/2}\xi,
  \]
  and when $\delta_\Omega(z)=\sqrt{ab}\eps_{j+1}$, we have
  \[
    \frac{\sigma_{\eps_{j+1}}(z)}{\eps_j^s}=a^s\frac{\sigma_{\eps_{j+1}}(z)}{(\delta_\Omega(z))^s}\geq a^s\left(\frac{\sigma_{\eps_j}(z)}{(\delta_\Omega(z))^s}+\xi\right)=\frac{\sigma_{\eps_j}(z)}{\eps_j^s}+a^s\xi.
  \]
  Hence, $\sigma_s=\sigma_{\eps_j}$ in a neighborhood of $\{z\in\Omega:\delta_\Omega(z)=\sqrt{ab}\eps_j\}$, so $\sigma_s$ is smooth on $\Omega$.  Convexity of $\psi_\xi$ and \eqref{eq:sigma_eps_hessian_lower_bound} guarantee that
  \[
    i\ddbar\sigma_s\geq E_s(\delta_\Omega)^s\omega
  \]
  on $\Omega$.  We have $\sigma_s\leq -(\delta_\Omega)^s$ on $\Omega$ by \eqref{eq:sigma_eps_normalized_upper_bound}.  Clearly $\frac{\sigma_s}{(\delta_\Omega)^s}$ has a uniform lower bound on $K_\eps$, so there must exist $C_s>0$ such that $\sigma_s\geq -C_s(\delta_\Omega)^s$.  Since the upper bound in \eqref{eq:sigma_eps_gradient_bound} is independent of $\eps$ and $\frac{\partial\psi_\xi}{\partial x}(x,y)+\frac{\partial\psi_\xi}{\partial y}(x,y)=1$, we have
  \[
    \abs{\nabla\sigma_s}\leq Ab^{1-s}G(b+1)^\eta(\delta_\Omega)^{s-1}+E_s b^s a^{-s} r(\delta_\Omega)^s\text{ on }\Omega.
  \]
  If we define $\rho_s=-(-\sigma_s)^{1/s}$, then we immediately obtain all of the necessary properties except for the fact that it is Lipschitz.  To check this, we use $(-\sigma_s)^{(1-s)/s}\leq C_s^{(1-s)/s}(\delta_\Omega)^{1-s}$ on $\Omega$ to show
  \[
    \abs{\nabla\rho_s}\leq \frac{1}{s}(-\sigma_s)^{(1-s)/s}\abs{\nabla\sigma_s}\leq\frac{1}{s}C_s^{(1-s)/s}\left(Ab^{1-s}G(b+1)^\eta+E_s b^s a^{-s} r(\delta_\Omega)\right),
  \]
  so $\rho_s$ has a uniformly bounded gradient.  Since $b\Omega$ is not necessarily rectifiable, this does not necessarily imply that $\rho_s$ is Lipschitz.  Let $z,w\in\Omega$.  If $|z-w|<\delta_\Omega(z)+\delta_\Omega(w)$, then the line segment connecting $z$ to $w$ must lie entirely in $\Omega$, so
  \[
    \abs{\rho_s(z)-\rho_s(w)}\leq |z-w|\sup_\Omega|\nabla\rho_s|.
  \]
  If $|z-w|\geq\delta_\Omega(z)+\delta_\Omega(w)$, then
  \[
    \abs{\rho_s(z)-\rho_s(w)}\leq C_s^{1/s}(\delta_\Omega(z)+\delta_\Omega(w))\leq C_s^{1/s}|z-w|.
  \]
  In either case, we have a uniform bound on the Lipschitz constant, so $\rho_s$ is Lipschitz on $\overline\Omega$.

\end{proof}

\begin{cor}
\label{cor:weak_DF_index}
  Let $\Omega\subset\mathbb{C}^n$ be a bounded pseudoconvex domain.  Then $DF_w(\Omega)=DF_{\Lambda^1}(\Omega)$.
\end{cor}

\begin{proof}
  Clearly $DF_w(\Omega)\geq DF_{\Lambda^1}(\Omega)$.  Suppose there exists $DF_w(\Omega)>\eta>DF_{\Lambda^1}(\Omega)$.  By definition, there exists $\lambda_\eta$ satisfying the hypotheses of Proposition \ref{prop:weak_DF_Index}.  Fix $DF_{\Lambda^1}(\Omega)<s<\eta$ and let $\rho_s$ be given by Proposition \ref{prop:weak_DF_Index}.  Then the existence of $\rho_s$ implies $s\leq DF_{\Lambda^1}(\Omega)$, a contradiction.
\end{proof}

\section{The Distance Function on $C^2$ Domains}

Throughout this paper, we use the signed distance function $\tilde\delta_\Omega$ as our canonical defining function.  While there are many advantages to working with an arbitrary defining function as in \cite{Ran81} or \cite{Liu17a,Liu19}, this necessitates working on a $C^3$ domain in order to study the rate at which the hessian changes across the level curves of the defining function.  With the signed distance function, we have Weinstock's formula \eqref{eq:Weinstock}, which gives us the same information using only second derivatives of the defining function.  We will see that this is essential when working on $C^2$ domains.

We will now outline the key results for the signed distance function on $C^2$ domains.  For a domain $\Omega\subset\mathbb{R}^n$ with $C^2$ boundary, there exists a neighborhood $U$ of $b\Omega$ such that for every point $x\in U$ there exists a unique point $\xi_\Omega(x)\in b\Omega$ such that $\delta_\Omega(x)=|x-\xi_\Omega(x)|$ (this follows from Theorem 4.18 in \cite{Fed59}).  We may assume that $U$ is sufficiently small so that $\tilde\delta_\Omega$ is $C^2$ on $U$ by a result of Krantz and Parks \cite{KrPa81}.  In particular, \eqref{eq:projection} holds on $U$.  For $x\in U$, we easily check that $\xi_\Omega((1-t)\xi_\Omega(x)+t x)=\xi_\Omega(x)$ and $\delta_\Omega((1-t)\xi_\Omega(x)+t x)=t\delta_\Omega(x)$ for all $0\leq t\leq 1$, so \eqref{eq:projection} gives us $\nabla\delta_\Omega((1-t)\xi_\Omega(x)+t x)=\nabla\delta_\Omega(x)$ for all $0<t\leq 1$.  In particular, taking a limit as $t\rightarrow 0^+$, we see that
\begin{equation}
\label{eq:gradient_projection}
  \nabla\tilde\delta_\Omega(x)=\nabla\tilde\delta_\Omega(\xi_\Omega(x))
\end{equation}
on $U$.  If we use $\nabla^2\tilde\delta_\Omega(x)$ to denote the Hessian of $\tilde\delta_\Omega$, $I$ to denote the identity matrix, and treat the gradient as a column vector, we may differentiate \eqref{eq:projection} to compute the Jacobian matrix
\begin{equation}
\label{eq:projection_jacobian}
  \nabla\xi_\Omega(x)=I-\tilde\delta_\Omega(x)\nabla^2\tilde\delta_\Omega(x)-(\nabla\tilde\delta_\Omega(x))(\nabla\tilde\delta_\Omega(x))^T
\end{equation}
for all $x\in U$.  Since $\nabla\tilde\delta_\Omega\cdot\nabla\tilde\delta_\Omega=1$ on $U$, we may also differentiate this to obtain
\begin{equation}
\label{eq:hessian_normal_vanishes}
  (\nabla^2\tilde\delta_\Omega(x))(\nabla\tilde\delta_\Omega(x))=0
\end{equation}
for all $x\in U$.  Using \eqref{eq:gradient_projection}, \eqref{eq:projection_jacobian}, and \eqref{eq:hessian_normal_vanishes} to differentiate \eqref{eq:gradient_projection}, we find that
\[
  \nabla^2\tilde\delta_\Omega(x)=\nabla^2\tilde\delta_\Omega(\xi_\Omega(x))\left(I-\tilde\delta_\Omega(x)\nabla^2\tilde\delta_\Omega(x)\right)
\]
on $U$.  On $U$, we may use linear algebra to solve this for $\nabla^2\tilde\delta_\Omega(x)$ to obtain Weinstock's formula \cite{Wei75} (see also \cite{HeMc12} and \cite{HaRa13} for further exposition on this formula):
\[
  \nabla^2\tilde\delta_\Omega(x)=\left(I+\tilde\delta_\Omega(x)\nabla^2\tilde\delta_\Omega(\xi_\Omega(x))\right)^{-1}\nabla^2\tilde\delta_\Omega(\xi_\Omega(x)).
\]
For our purposes, it will suffice to compute the low order approximation
\begin{equation}
\label{eq:Weinstock}
  \abs{\nabla^2\tilde\delta_\Omega(x)-\nabla^2\tilde\delta_\Omega(\xi_\Omega(x))+\tilde\delta_\Omega(x)\left(\nabla^2\tilde\delta_\Omega(\xi_\Omega(x))\right)^2}\leq O((\delta_\Omega(x))^2)
\end{equation}
for all $x\in U$.

Our key results will rely on comparing the signed distance function to a defining function with better potential theoretic properties.  For this purpose, the following lemma will be crucial:
\begin{lem}
\label{lem:defining_function_comparison}
  Let $\Omega\subset\mathbb{R}^n$ be a bounded domain with a $C^k$ boundary, $k\geq 1$, and let $\rho_1$ and $\rho_2$ be $C^k$ defining functions for $\Omega$ on some neighborhood $U$ of $b\Omega$.  Then $\rho_1=h\rho_2$ for some positive function $h\in C^{k-1}(U)\cap C^k(U\backslash b\Omega)$ satisfying
  \begin{equation}
  \label{eq:h_hessian_vanishing}
    \lim_{\eps\rightarrow 0^+}\eps\norm{h}_{C^k(U_\eps)}=0,
  \end{equation}
  where $U_\eps=\set{x\in U:\delta_\Omega(x)>\eps}$ for any $\eps>0$.
\end{lem}

\begin{proof}
  That $h$ exists and $h\in C^{k-1}(U)$ is well known (see Lemma 1.1.3 in \cite{ChSh01}, for example).  That $h\in C^k(U\backslash b\Omega)$ is easily confirmed since $h=\frac{\rho_1}{\rho_2}$ on $U\backslash b\Omega$.

  For a function $f$ that is $C^\ell$ in a neighborhood of a point $p\in\mathbb{R}^n$ for some $\ell\geq 0$, we let $P^\ell_{f,p}(x)$ denote the $\ell$th Taylor Polynomial for $f$ centered at $p$.  For $p\in b\Omega$, $\rho_2(p)=0$ implies that
  \[
    E_{p}(x)=\rho_1(x)-P^{k-1}_{h,p}(x)\rho_2(x)
  \]
  is an element of $C^k(U)$ with a vanishing $k$th Taylor polynomial centered at $p$.  As a result, for any $0\leq\ell\leq k$ and differential operator $D^\ell$ of order $\ell$ defined on $U$, $|D^\ell E_p(x)|\leq o(|x-p|^{k-\ell})$ as $x\rightarrow p$.  Let $U_0$ be a neighborhood of $b\Omega$ that is relatively compact in $U$.  If we set $F_{p,D^\ell}(x)=\frac{D^\ell E_p(x)}{|x-p|^{k-\ell}}$ when $x\neq p$ and $F_{p,D^\ell}(p)=0$, then $F_{p,D^\ell}(x)$ is a continuous function for $(p,x)\in b\Omega\times\overline{U_0}$.  Since $b\Omega\times\overline{U_0}$ is compact, $|F_{p,D^\ell}(x)|\leq f_{D^\ell}(|x-p|)$ on $b\Omega\times\overline{U_0}$ for a continuous function $f_{D^\ell}:\mathbb{R}\rightarrow\mathbb{R}$ that vanishes at $0$.  Hence, for every $0\leq\ell\leq k$ and differential operator $D^\ell$ of order $\ell$ there exists a continuous function $f_{D^\ell}$ on $\mathbb{R}$ vanishing at $0$ such that
  \begin{equation}
  \label{eq:derivative_error_estimate}
    |D^\ell E_p(x)|\leq |x-p|^{k-\ell} f_{D^\ell}(|x-p|)
  \end{equation}
  whenever $p\in b\Omega$ and $x\in\overline{U_0}$.

  On $U\backslash b\Omega$
  \[
    h(x)=\frac{\rho_1(x)}{\rho_2(x)}=\frac{P^{k-1}_{h,p}(x)\rho_2(x)+E_p(x)}{\rho_2(x)}=P^{k-1}_{h,p}(x)+\frac{E_p(x)}{\rho_2(x)}.
  \]
  For $p\in b\Omega$ and $M>0$ sufficiently large,
  \[
    \Gamma_{p,M}=\set{x\in U:|x-p|\leq M|\rho_2(x)|}
  \]
  defines a non-tangential approach region for $p$.  On $\Gamma_{p,M}$, \eqref{eq:derivative_error_estimate} implies that for any $k$th order differential operator $D^k$ we have
  \[
    |D^k h(x)|=\abs{D^k\frac{E_p(x)}{\rho_2(x)}}\leq o\left(\frac{1}{|x-p|}\right),
  \]
  with $\lim_{x\rightarrow p}|x-p||D^k h(x)|=0$ uniformly in $x$ and $p$.
\end{proof}

\section{The Strong Diederich-Forn{\ae}ss Index on $C^2$ Domains}

We begin by defining some key hermitian invariants.  Let $\Omega\subset\mathbb{C}^n$ be a domain with $C^2$ boundary, and let $U$ be a neighborhood of $b\Omega$ on which $\tilde\delta_\Omega$ is $C^2$.  We define the real one-form
\[
  \alpha_\Omega=4\sum_{j=1}^n\re\left(\frac{\partial\tilde\delta_\Omega}{\partial\bar z_j}\dbar\left(\frac{\partial\tilde\delta_\Omega}{\partial z_j}\right)\right)
\]
on $U$.  Using (5.85) in \cite{Str10} and the fact that $\abs{\partial\tilde\delta_\Omega}^2=\frac{1}{2}$ on $U$, $\alpha_\Omega|_{b\Omega}$ agrees with D'Angelo's one-form (see 5.9 in \cite{Str10} for further background on this one-form).  While $\alpha_\Omega$ itself is only a hermitian invariant, the cohomology class represented by the restriction of $\alpha_\Omega$ to any complex submanifold in $b\Omega$ is an important biholomorphic invariant \cite{BoSt93}.

For $z\in U$, let $\nu_z\in T^{1,0}_z$ denote the unique vector satisfying $\partial\tilde\delta_\Omega(\nu_z)=1$ and $|\nu_z|=\frac{1}{|\partial\tilde\delta_\Omega|}=\sqrt{2}$, i.e.,
\begin{equation}
\label{eq:nu_characterization}
  \nu_z=4\sum_{j=1}^n\frac{\partial\tilde\delta_\Omega}{\partial\bar z_j}\frac{\partial}{\partial z_j}.
\end{equation}
Then for any $\tau_z\in T^{1,0}_z$, we have
\begin{equation}
\label{eq:alpha_characterization}
  \alpha_\Omega(\tau_z)=\frac{i}{2}\ddbar\tilde\delta_\Omega(\tau_z,\bar\nu_z).
\end{equation}

Let $\pi_{1,0}\alpha_\Omega$ (resp. $\pi_{0,1}\alpha_\Omega$) denote the projection of $\alpha_\Omega$ onto its $(1,0)$ component (resp. $(0,1)$-component).  On $U$, we also define a real, positive semi-definite $(1,1)$-form
\[
  \beta_\Omega=\sum_{j=1}^n i\left(\partial\frac{\partial\tilde\delta}{\partial z_j}\wedge\dbar\frac{\partial\tilde\delta}{\partial \bar z_j}+\partial\frac{\partial\tilde\delta}{\partial\bar z_j}\wedge\dbar\frac{\partial\tilde\delta}{\partial z_j}\right)-2i\pi_{1,0}\alpha_\Omega\wedge\pi_{0,1}\alpha_\Omega.
\]
Once again, $\beta_\Omega$ is only a hermitian invariant.  If we restrict $\beta_\Omega$ to $\beta_\Omega(\tau,\bar\tau)$ at $z\in b\Omega$, where $\tau\in T^{1,0}_z$ lies in the kernel of the Levi-form at $z$, then $\beta_\Omega$ agrees with the definition given in \cite{Har19}.  Section 3 of \cite{Har19} contains further background on the relationship between $\alpha_\Omega$ and $\beta_\Omega$, which will not be relevant for the present paper.  We note that the semi-definite nature of $\beta_\Omega$ can be confirmed by rotating coordinates so that for $z\in U$, $\partial\tilde\delta_\Omega|_z=-\frac{i}{2}dz_n$.  Then $\alpha_\Omega|_z=-2\im\left(\dbar\left(\frac{\partial\tilde\delta_\Omega}{\partial z_n}\right)|_z\right)$. Since \eqref{eq:hessian_normal_vanishes} implies $\alpha_\Omega|_z=-2\im\left(\dbar\left(\frac{\partial\tilde\delta_\Omega}{\partial\bar z_n}\right)|_z\right)$ as well, we have
\[
  \beta_\Omega|_z=\sum_{j=1}^{n-1} i\left(\partial\frac{\partial\tilde\delta}{\partial z_j}\wedge\dbar\frac{\partial\tilde\delta}{\partial \bar z_j}+\partial\frac{\partial\tilde\delta}{\partial\bar z_j}\wedge\dbar\frac{\partial\tilde\delta}{\partial z_j}\right),
\]
which is clearly positive semi-definite.

The utility of $\alpha_\Omega$ and $\beta_\Omega$ can be seen by considering \eqref{eq:Weinstock} in complex coordinates.  If we write $z_j=x_j+iy_j$, \eqref{eq:Weinstock} implies
\begin{multline*}
  i\ddbar\tilde\delta_\Omega|_z
  = i\ddbar\tilde\delta_\Omega|_{\xi_\Omega(z)}
  +\sum_{j=1}^n(-\tilde\delta_\Omega(z))i\left(\partial\frac{\partial\tilde\delta_\Omega}{\partial x_j}\wedge\dbar\frac{\partial\tilde\delta_\Omega}{\partial x_j}+\partial\frac{\partial\tilde\delta_\Omega}{\partial y_j}\wedge\dbar\frac{\partial\tilde\delta_\Omega}{\partial y_j}\right)\bigg|_{\xi_\Omega(z)}\\
  +O((\delta_\Omega(z))^2)
\end{multline*}
for every $z\in U$.  Expressing this in complex coordinates, we find that
\begin{equation}
\label{eq:Weinstock_complex}
  i\ddbar\tilde\delta_\Omega|_z\\
  = i\ddbar\tilde\delta_\Omega|_{\xi_\Omega(z)}+2(-\tilde\delta_\Omega(z))\left(\beta_\Omega+2i\pi_{1,0}\alpha_\Omega\wedge\pi_{0,1}\alpha_\Omega\right)\bigg|_{\xi_\Omega(z)}+O((\delta_\Omega(z))^2)
\end{equation}

Our first lemma provides the key identity relating our plurisubharmonic functions to the appropriate defining function.

\begin{lem}
\label{lem:lambda_determinant}
  Let $\Omega\subset\mathbb{C}^n$ be a bounded domain with $C^2$ boundary.  Suppose that there exists a neighborhood of $b\Omega$ denoted $U$ and a function $\varphi\in C^2(U\backslash b\Omega)\cap C^1(U)$ such that
  \begin{equation}
  \label{eq:varphi_hessian_vanishing}
    \lim_{\eps\rightarrow 0^+}\eps\norm{\varphi}_{C^2(\set{z\in\Omega\cap U:\delta_\Omega(z)>\eps})}=0.
  \end{equation}
  For $0<\eta<1$ and $t\in\mathbb{R}$, on $U\cap\Omega$ we set $\lambda_\eta=-\left(-e^{-t|z|^2-\varphi}\tilde\delta_\Omega\right)^\eta$.  Given $M_\eta\in\mathbb{R}$, there exists a neighborhood $U_\eta$ of $b\Omega$ such that $\tilde\delta_\Omega$ is $C^2$ on $U_\eta$ and $i\ddbar\lambda_\eta(\nu,\bar\nu)-2M_\eta(-\lambda_\eta)>0$ on $U_\eta\cap\Omega$.  Furthermore, for any $\tau\in T^{1,0}(U_\eta)$ such that $\partial\tilde\delta_\Omega(\tau)\equiv 0$, we have
  \begin{multline}
  \label{eq:lambda_determinant}
    \frac{1}{\eta(-\lambda_\eta(z))}\left(i\ddbar\lambda_\eta(\tau,\bar\tau)|_z-M_\eta(-\lambda_\eta(z))|\tau|^2-\frac{\abs{i\ddbar\lambda_\eta(\tau,\bar\nu)|_z}^2}{i\ddbar\lambda_\eta(\nu,\bar\nu)|_z-2M_\eta(-\lambda_\eta(z))}\right)\\
    =(2t-\eta^{-1}M_\eta)\abs{\tau}^2+i\ddbar\varphi(\tau,\bar\tau)|_z+(-\tilde\delta(z))^{-1}i\ddbar\tilde\delta(\tau,\bar\tau)|_{\xi_\Omega(z)}+2\beta(\tau,\bar\tau)|_{\xi_\Omega(z)}\\
    -\frac{\eta}{1-\eta}\abs{\left(t\partial|z|^2(\tau)+\partial\varphi(\tau)\right)|_z-2\alpha(\tau)|_{\xi_\Omega(z)}}^2+\Theta_\eta(\tau,\bar\tau),
  \end{multline}
  where $\Theta_\eta$ is a real $(1,1)$-form on $U_\eta\cap\Omega$ with continuous coefficients that vanish on $b\Omega$.
\end{lem}

\begin{proof}
  We assume that we have already restricted to a neighborhood $U_\eta$ of $b\Omega$, to be determined later, on which $\tilde\delta_\Omega$ is $C^2$.  Since $\log(-\lambda_\eta)=\eta\left(-t|z|^2-\varphi+\log(-\tilde\delta_\Omega)\right)$, we have
  \[
    \partial\lambda_\eta=\eta(-\lambda_\eta)\left(t\partial|z|^2+\partial\varphi+(-\tilde\delta_\Omega)^{-1}\partial\tilde\delta_\Omega\right),
  \]
  so
  \begin{equation}
  \label{eq:lambda_tau}
    \partial\lambda_\eta(\tau)=\eta(-\lambda_\eta)\left(t\partial|z|^2(\tau)+\partial\varphi(\tau)\right)
  \end{equation}
  and
  \begin{equation}
  \label{eq:lambda_nu}
    \abs{\partial\lambda_\eta(\nu)-\eta(-\lambda_\eta)(-\tilde\delta_\Omega)^{-1}}\leq O((-\lambda_\eta)).
  \end{equation}
  Furthermore, we have
  \begin{multline*}
    i\ddbar\lambda_\eta=-(-\lambda_\eta)^{-1}i\partial\lambda_\eta\wedge\dbar\lambda_\eta\\
    +\eta(-\lambda_\eta)\left(2t\omega+i\ddbar\varphi+(-\tilde\delta_\Omega)^{-1}i\ddbar\tilde\delta_\Omega+(-\tilde\delta_\Omega)^{-2}i\partial\tilde\delta_\Omega\wedge\dbar\tilde\delta_\Omega\right).
  \end{multline*}
  We first confirm that
  \[
    \abs{i\ddbar\lambda_\eta(\nu,\bar\nu)+(-\lambda_\eta)^{-1}\abs{\partial\lambda_\eta(\nu)}^2-\eta(-\lambda_\eta)(-\tilde\delta_\Omega)^{-2}}\leq O((-\lambda_\eta)(-\tilde\delta_\Omega)^{-1}).
  \]
  Using \eqref{eq:lambda_nu}, we have
  \begin{equation}
  \label{eq:lambda_nu_nu}
    \abs{i\ddbar\lambda_\eta(\nu,\bar\nu)-\eta(1-\eta)(-\lambda_\eta)(-\tilde\delta_\Omega)^{-2}}\leq O((-\lambda_\eta)(-\tilde\delta_\Omega)^{-1}).
  \end{equation}
  Since $0<\eta<1$ and $2M_\eta(-\lambda_\eta)$ is bounded by the error term, we may choose $U_\eta$ sufficiently small so that $i\ddbar\lambda_\eta(\nu,\bar\nu)-2M_\eta(-\lambda_\eta)>0$ on $U_\eta\cap\Omega$.

  Using \eqref{eq:alpha_characterization}, we have
  \[
    i\ddbar\lambda_\eta(\tau,\bar\nu)=-(-\lambda_\eta)^{-1}\partial\lambda_\eta(\tau)\overline{\partial\lambda_\eta(\nu)}+\eta(-\lambda_\eta)\left(i\ddbar\varphi(\tau,\bar\nu)+2(-\tilde\delta_\Omega)^{-1}\alpha_\Omega(\tau)\right).
  \]
  From \eqref{eq:varphi_hessian_vanishing}, we know that $\lim_{z\rightarrow b\Omega}(-\tilde\delta_\Omega(z))i\ddbar\varphi(\tau,\bar\nu)|_z=0$ with uniform convergence, which we will denote $\abs{i\ddbar\varphi(\tau,\bar\nu)}\leq o((-\tilde\delta_\Omega)^{-1}|\tau|)$.  Using \eqref{eq:lambda_nu}, we have
  \begin{equation}
  \label{eq:lambda_tau_nu}
    \abs{i\ddbar\lambda_\eta(\tau,\bar\nu)+\eta(-\tilde\delta_\Omega)^{-1}\partial\lambda_\eta(\tau)-2\eta(-\lambda_\eta)(-\tilde\delta_\Omega)^{-1}\alpha_\Omega(\tau)}\leq o\left((-\lambda_\eta)(-\tilde\delta_\Omega)^{-1}|\tau|\right).
  \end{equation}
  Combining \eqref{eq:lambda_tau_nu} with \eqref{eq:lambda_nu_nu}, we obtain
  \begin{equation}
  \label{eq:lambda_fraction}
    \abs{\frac{\abs{i\ddbar\lambda_\eta(\tau,\bar\nu)}^2}{i\ddbar\lambda_\eta(\nu,\bar\nu)-2M_\eta(-\lambda_\eta)}-\frac{\eta\abs{\partial\lambda_\eta(\tau)-2(-\lambda_\eta)\alpha_\Omega(\tau)}^2}{(1-\eta)(-\lambda_\eta)}}\leq o\left((-\lambda_\eta)|\tau|^2\right).
  \end{equation}

  Finally, we have
  \begin{multline*}
    i\ddbar\lambda_\eta(\tau,\bar\tau)=-(-\lambda_\eta)^{-1}\abs{\partial\lambda_\eta(\tau)}^2\\
    +\eta(-\lambda_\eta)\left(2t\abs{\tau}^2+i\ddbar\varphi(\tau,\bar\tau)+(-\tilde\delta_\Omega)^{-1}i\ddbar\tilde\delta_\Omega(\tau,\bar\tau)\right).
  \end{multline*}
  From \eqref{eq:Weinstock_complex}, we obtain for any $z\in U_\eta\cap\Omega$,
  \begin{multline}
  \label{eq:lambda_tau_tau}
    \Big|i\ddbar\lambda_\eta(\tau,\bar\tau)|_z+(-\lambda_\eta(z))^{-1}\abs{\partial\lambda_\eta(\tau)|_z}^2-\eta(-\lambda_\eta(z))\left(2t\abs{\tau}^2+i\ddbar\varphi(\tau,\bar\tau)|_z\right)\\
    -\eta(-\lambda_\eta(z))\left((-\tilde\delta_\Omega(z))^{-1}i\ddbar\tilde\delta_\Omega(\tau,\bar\tau)|_{\xi_\Omega(z)}+2(\beta_\Omega(\tau,\bar\tau)+2\abs{\alpha_\Omega(\tau)}^2)|_{\xi_\Omega(z)}\right)\Big|\\
    \leq O\left((-\lambda_\eta(z))(-\tilde\delta_\Omega(z))|\tau|^2\right).
  \end{multline}
  To combine \eqref{eq:lambda_tau_tau} with \eqref{eq:lambda_fraction}, we first compute
  \begin{multline*}
    \frac{\eta\abs{\partial\lambda_\eta(\tau)-2(-\lambda_\eta)\alpha_\Omega(\tau)}^2}{(1-\eta)(-\lambda_\eta)}+(-\lambda_\eta)^{-1}\abs{\partial\lambda_\eta(\tau)}^2-4\eta(-\lambda_\eta)\abs{\alpha_\Omega(\tau)}^2\\
    =\frac{\abs{\partial\lambda_\eta(\tau)-2\eta(-\lambda_\eta)\alpha_\Omega(\tau)}^2}{(1-\eta)(-\lambda_\eta)}.
  \end{multline*}
  Since \eqref{eq:Weinstock_complex} also implies $\abs{\alpha_\Omega(\tau)|_z-\alpha_\Omega(\tau)_{\xi_\Omega(z)}}\leq O((-\tilde\delta_\Omega(z))|\tau|)$, we may substitute this in \eqref{eq:lambda_fraction} and combine this with \eqref{eq:lambda_tau_tau} to obtain
  \begin{multline*}
    \Bigg|i\ddbar\lambda_\eta(\tau,\bar\tau)|_z-\frac{\abs{i\ddbar\lambda_\eta(\tau,\bar\nu)}^2}{i\ddbar\lambda_\eta(\nu,\bar\nu)-2M_\eta(-\lambda_\eta)}|_z\\
    +\frac{\abs{\partial\lambda_\eta(\tau)|_z-2\eta(-\lambda_\eta(z))\alpha_\Omega(\tau)|_{\xi_\Omega(z)}}^2}{(1-\eta)(-\lambda_\eta(z))}\\
    -\eta(-\lambda_\eta(z))\left(2t\abs{\tau}^2+i\ddbar\varphi(\tau,\bar\tau)|_z+(-\tilde\delta_\Omega(z))^{-1}i\ddbar\tilde\delta_\Omega(\tau,\bar\tau)|_{\xi_\Omega(z)}\right)\\
    -\eta(-\lambda_\eta(z))2\beta_\Omega(\tau,\bar\tau)|_{\xi_\Omega(z)}\Bigg|\leq o\left((-\lambda_\eta(z))|\tau|^2\right).
  \end{multline*}
  If we substitute \eqref{eq:lambda_tau} in this, \eqref{eq:lambda_determinant} will follow.

\end{proof}

Now we are ready to characterize $DF_{C^2}(\Omega)$ when $\Omega$ only has a $C^2$ boundary.
\begin{prop}
\label{prop:C2_equivalence}
  Let $\Omega\subset\mathbb{C}^n$ be a bounded pseudoconvex domain with $C^2$ boundary, and let $r=\sup_\Omega|z|$.
  \begin{enumerate}
    \item Suppose that for some $0<\eta<1$ and $M_\eta\in\mathbb{R}$ there exists a $C^2$ defining function $\rho_\eta$ for $\Omega$ with the property that $\lambda_\eta=-(-\rho_\eta)^\eta$ satisfies $i\ddbar\lambda_\eta\geq M_\eta(-\lambda_\eta)\omega$ on $U_\eta\cap\Omega$ for some neighborhood $U_\eta$ of $b\Omega$.  Then for every $0<s<\eta$ and $N_s<\frac{1}{2r^2}\left(\frac{1}{s}-\frac{1}{\eta}\right)+\eta^{-1}M_\eta$ there exists a neighborhood $U_s\subset U_\eta$ of $b\Omega$ and a function $\varphi_s\in C^1(U_s)\cap C^2(U_s\backslash b\Omega)$ satisfying \eqref{eq:varphi_hessian_vanishing} such that for every $z\in U_s\cap\Omega$ and $\tau\in T^{1,0}_z$ satisfying $\partial\tilde\delta_\Omega(\tau)=0$, we have
        \begin{multline}
        \label{eq:varphi_hessian_estimate_conclusion}
          N_s\abs{\tau}^2+\frac{s}{1-s}\abs{\partial\varphi_s(\tau)|_z-2\alpha(\tau)|_{\xi_\Omega(z)}}^2\leq\\ i\ddbar\varphi_s(\tau,\bar\tau)|_z+(-\tilde\delta(z))^{-1}i\ddbar\tilde\delta(\tau,\bar\tau)|_{\xi_\Omega(z)}+2\beta(\tau,\bar\tau)|_{\xi_\Omega(z)}.
        \end{multline}

    \item Suppose that for some $0<s<1$ and $N_s\in\mathbb{R}$ there exists a neighborhood $U_s$ of $b\Omega$ and a function $\varphi_s\in C^1(U_s)\cap C^2(U_s\backslash b\Omega)$ satisfying \eqref{eq:varphi_hessian_vanishing} such that for every $z\in U_s\cap\Omega$ and $\tau\in T^{1,0}_z$ satisfying $\partial\tilde\delta_\Omega(\tau)=0$, we have
        \begin{multline}
        \label{eq:varphi_hessian_estimate_hypothesis}
          N_s\abs{\tau}^2+\frac{s}{1-s}\abs{\partial\varphi_s(\tau)|_z-2\alpha(\tau)|_{\xi_\Omega(z)}}^2\leq \\
          i\ddbar\varphi_s(\tau,\bar\tau)|_z
          +(-\tilde\delta(z))^{-1}i\ddbar\tilde\delta(\tau,\bar\tau)|_{\xi_\Omega(z)}+2\beta(\tau,\bar\tau)|_{\xi_\Omega(z)}.
        \end{multline}
        Then for every $0<\eta<s$ and $M_\eta<\frac{1}{2r^2}\left(1-\frac{\eta}{s}\right)+\eta N_s$ there exists a $C^2$ defining function $\rho_\eta$ for $\Omega$ such that $\lambda_\eta=-(-\rho_\eta)^\eta$ satisfies $i\ddbar\lambda_\eta\geq M_\eta(-\lambda_\eta)\omega$ on $U_\eta\cap\Omega$ for some neighborhood $U_\eta$ of $b\Omega$.
  \end{enumerate}
\end{prop}

\begin{proof}
To prove $(1)$, we may assume that $U_\eta$ is at least as small as the $U_\eta$ given by Lemma \ref{lem:lambda_determinant}.  Using Lemma \ref{lem:defining_function_comparison}, we write $\rho_\eta=h_\eta\tilde\delta_\Omega$ for $h_\eta\in C^1(U)\cap C^2(U\backslash b\Omega)$ satisfying \eqref{eq:h_hessian_vanishing}.  We set $t=-\frac{1}{2r^2}\left(\frac{1}{s}-\frac{1}{\eta}\right)<0$.  Since $h_\eta>0$ on $U_\eta$, we may define $\varphi_s(z)=-\log h_\eta(z)-t|z|^2$ on $U_\eta$.  \eqref{eq:varphi_hessian_vanishing} will follow from \eqref{eq:h_hessian_vanishing}.

For any $z\in U_\eta\cap\Omega$ and $\tau\in T^{1,0}_z$ satisfying $\partial\tilde\delta_\Omega(\tau)=0$, we set
\[
  L=\tau-\frac{i\ddbar\lambda_\eta(\tau,\nu)}{i\ddbar\lambda_\eta(\nu,\nu)-2M_\eta(-\lambda_\eta)}\nu.
\]
Then we have
\begin{multline*}
  0\leq i\ddbar\lambda_\eta(L,\bar L)-M_\eta(-\lambda_\eta)|L|^2\\
  =i\ddbar\lambda_\eta(\tau,\bar\tau)-M_\eta(-\lambda_\eta)|\tau|^2-\frac{\abs{i\ddbar\lambda_\eta(\tau,\bar\nu)}^2}{i\ddbar\lambda_\eta(\nu,\bar\nu)-2M_\eta(-\lambda_\eta)}.
\end{multline*}
By \eqref{eq:lambda_determinant}, we have
\begin{multline*}
  0\leq (2t-\eta^{-1}M_\eta)\abs{\tau}^2+i\ddbar\varphi_s(\tau,\bar\tau)|_z+(-\tilde\delta_\Omega(z))^{-1}i\ddbar\tilde\delta_\Omega(\tau,\bar\tau)|_{\xi_\Omega(z)}\\
  +2\beta_\Omega(\tau,\bar\tau)|_{\xi_\Omega(z)}
  -\frac{\eta}{1-\eta}\abs{\left(t\partial|z|^2(\tau)+\partial\varphi_s(\tau)\right)|_z-2\alpha_\Omega(\tau)|_{\xi_\Omega(z)}}^2+\Theta_\eta(\tau,\bar\tau),
\end{multline*}

Note that $\abs{t\partial|z|^2(\tau)}\leq-\sqrt{2}r t|\tau|$.  If we set $\epsilon=\frac{(1-\eta)s}{\eta(1-s)}$, then $\frac{\epsilon}{1-\epsilon}=\frac{(1-\eta)s}{\eta-s}$, so we may use \eqref{eq:basic_inequality} to obtain
\begin{multline*}
  \abs{\left(t\partial|z|^2(\tau)+\partial\varphi_s(\tau)\right)|_z-2\alpha_\Omega(\tau)|_{\xi_\Omega(z)}}^2\\
  \geq \frac{(1-\eta)s}{\eta(1-s)}\abs{\partial\varphi_s(\tau)|_z-2\alpha_\Omega(\tau)|_{\xi_\Omega(z)}}^2-\frac{(1-\eta)s}{\eta-s}2r^2 t^2|\tau|^2.
\end{multline*}
so we obtain
\begin{multline*}
  0\leq -\left(\frac{1}{2r^2}\left(\frac{1}{s}-\frac{1}{\eta}\right)+\eta^{-1}M_\eta\right)\abs{\tau}^2
  +i\ddbar\varphi_s(\tau,\bar\tau)|_z\\
  +(-\tilde\delta_\Omega(z))^{-1}i\ddbar\tilde\delta_\Omega(\tau,\bar\tau)|_{\xi_\Omega(z)}
  +2\beta_\Omega(\tau,\bar\tau)|_{\xi_\Omega(z)}\\
  -\frac{s}{1-s}\abs{\partial\varphi_s(\tau)|_z-2\alpha_\Omega(\tau)|_{\xi_\Omega(z)}}^2+\Theta_\eta(\tau,\bar\tau).
\end{multline*}
If we choose $U_s$ sufficiently small so that
\[
  \Theta_\eta(\tau,\bar\tau)\leq\left(\frac{1}{2r^2}\left(\frac{1}{s}-\frac{1}{\eta}\right)+\eta^{-1}M_\eta-N_s\right)|\tau|^2
\]
on $U_s$, then we will have \eqref{eq:varphi_hessian_estimate_conclusion}.

To prove $(2)$, we let $\rho_\eta=e^{-t|z|^2-\varphi_s}\tilde\delta_\Omega$ for $t=\frac{1}{2r^2}\left(\frac{1}{\eta}-\frac{1}{s}\right)>0$.  As before, we assume that $U_\eta\subset U_s$ has been chosen sufficiently small so that Lemma \ref{lem:lambda_determinant} applies.

Fix $z\in U_\eta$ and let $L\in T^{1,0}_z$.  Set $\tau=L-\partial\tilde\delta_\Omega(L)\nu$, so that $\partial\tilde\delta_\Omega(\tau)=0$.  Set $\epsilon=\frac{\eta(1-s)}{(1-\eta)s}$, so that $\frac{\epsilon}{1-\epsilon}=\frac{\eta(1-s)}{s-\eta}$.  We may use \eqref{eq:basic_inequality} to obtain
\begin{multline*}
  \abs{\partial\varphi_s(\tau)|_z-2\alpha_\Omega(\tau)|_{\xi_\Omega(z)}}^2\\
  \geq\frac{\eta(1-s)}{(1-\eta)s}\abs{\left(t\partial|z|^2(\tau)+\partial\varphi_s(\tau)\right)|_z-2\alpha_\Omega(\tau)|_{\xi_\Omega(z)}}^2-\left(\frac{\eta(1-s)}{s-\eta}\right)2r^2t^2|\tau|^2.
\end{multline*}
Substituting this into \eqref{eq:varphi_hessian_estimate_hypothesis}, we have
\begin{multline*}
  \left(N_s-\left(\frac{\eta s}{s-\eta}\right)2r^2t^2\right)\abs{\tau}^2+\frac{\eta}{1-\eta}\abs{\left(t\partial|z|^2(\tau)+\partial\varphi_s(\tau)\right)|_z-2\alpha_\Omega(\tau)|_{\xi_\Omega(z)}}^2\leq \\
  i\ddbar\varphi_s(\tau,\bar\tau)|_z
  +(-\tilde\delta_\Omega(z))^{-1}i\ddbar\tilde\delta_\Omega(\tau,\bar\tau)|_{\xi_\Omega(z)}+2\beta_\Omega(\tau,\bar\tau)|_{\xi_\Omega(z)}.
\end{multline*}
Combining this with \eqref{eq:lambda_determinant}, we obtain
\begin{multline*}
  \frac{1}{\eta(-\lambda_\eta(z))}\left(i\ddbar\lambda_\eta(\tau,\bar\tau)|_z-M_\eta(-\lambda_\eta(z))|\tau|^2-\frac{\abs{i\ddbar\lambda_\eta(\tau,\bar\nu)|_z}^2}{i\ddbar\lambda_\eta(\nu,\bar\nu)|_z-2M_\eta(-\lambda_\eta(z))}\right)\\
  \geq\left(\frac{1}{2r^2}\left(\frac{1}{\eta}-\frac{1}{s}\right)-\eta^{-1}M_\eta+N_s\right)\abs{\tau}^2+\Theta_\eta(\tau,\bar\tau),
\end{multline*}
If we choose $U_\eta$ sufficiently small so that
\[
  \Theta_\eta(\tau,\bar\tau)\geq-\left(\frac{1}{2r^2}\left(\frac{1}{\eta}-\frac{1}{s}\right)-\eta^{-1}M_\eta+N_s\right)|\tau|^2
\]
on $U_\eta\cap\Omega$, then we have
\begin{multline*}
  \frac{1}{\eta(-\lambda_\eta(z))}\left(i\ddbar\lambda_\eta(\tau,\bar\tau)|_z-M_\eta(-\lambda_\eta(z))|\tau|^2-\frac{\abs{i\ddbar\lambda_\eta(\tau,\bar\nu)|_z}^2}{i\ddbar\lambda_\eta(\nu,\bar\nu)|_z-2M_\eta(-\lambda_\eta(z))}\right)\\
  \geq 0
\end{multline*}
on $U_\eta\cap\Omega$.

Now,
\begin{multline*}
  i\ddbar\lambda_\eta(L,\bar L)-M_\eta(-\lambda_\eta)|L|^2\\
  =i\ddbar\lambda_\eta(\tau,\bar\tau)-M_\eta(-\lambda_\eta)|\tau|^2-2\re\left(i\ddbar\lambda_\eta(\tau,\bar\nu)\dbar\tilde\delta_\Omega(\bar L)\right)\\
  +(i\ddbar\lambda_\eta(\nu,\bar\nu)-2M_\eta(-\lambda_\eta))\abs{\partial\tilde\delta_\Omega(L)}^2.
\end{multline*}
Rearranging terms, we have
\begin{multline*}
  i\ddbar\lambda_\eta(L,\bar L)-M_\eta(-\lambda_\eta)|L|^2\\
  =i\ddbar\lambda_\eta(\tau,\bar\tau)-M_\eta(-\lambda_\eta)|\tau|^2-\frac{\abs{i\ddbar\lambda_\eta(\tau,\bar\nu)}^2}{i\ddbar\lambda_\eta(\nu,\bar\nu)-2M_\eta(-\lambda_\eta)}\\
  +(i\ddbar\lambda_\eta(\nu,\bar\nu)-2M_\eta(-\lambda_\eta))\abs{\partial\tilde\delta_\Omega(L)-\frac{i\ddbar\lambda_\eta(\tau,\bar\nu)}{i\ddbar\lambda_\eta(\nu,\bar\nu)-2M_\eta(-\lambda_\eta)}}^2.
\end{multline*}
Since the final term is positive on $U_\eta$, we have $i\ddbar\lambda_\eta(L,\bar L)-M_\eta(-\lambda_\eta)|L|^2\geq 0$ on $U_\eta\cap\Omega$.  Since $L$ was arbitrary, we are done.
\end{proof}

\section{Estimates for the Strong Diederich-Forn{\ae}ss Index}

\label{sec:estimates}

Proposition \ref{prop:C2_equivalence} allows us to provide a quantitative statement of Diederich and Forn{\ae}ss's original result.  It is of great interest to note that this lower bound for the Diederich-Forn{\ae}ss index depends entirely on the size of the form $\alpha_\Omega$ when restricted to the null-space of the Levi-form and the size of $\Omega$.  These are not biholomorphic invariants, but this should be expected since our method depends heavily on the function $|z|^2$ generating the K\"ahler form for the Euclidean metric.  It is known that $\alpha_\Omega$ is $d$-closed when restricted to a complex submanifold in the boundary \cite{BoSt93}.  When this restriction is also $d$-exact, this can be used to construct an improved weight function $\varphi$ and strengthen this result, as studied by the author in \cite{Har19}.
\begin{cor}
\label{cor:DF_universal}
  Let $\Omega\subset\mathbb{C}^n$ be a bounded pseudoconvex domain with $C^2$ boundary.  Let $r=\sup_\Omega|z|$.  For $p\in b\Omega$, define
  \[
    \mathcal{N}_p(\Omega)=\set{\tau\in T^{1,0}(b\Omega):i\ddbar\tilde\delta_\Omega(\tau,\bar\theta)|_p=0\text{ for all }\theta\in T^{1,0}(b\Omega)}
  \]
  and
  \[
    A=\sup_{\set{p\in b\Omega:\mathcal{N}_p(\Omega)\neq\emptyset}}\sup_{\tau\in\mathcal{N}_p(\Omega)\backslash\{0\}}\frac{\abs{\alpha_\Omega(\tau)}}{\abs{\tau}},
  \]
  unless $\Omega$ is strictly pseudoconvex, in which case $A=0$.
  Then
  \[
    DF_{C^2}(\Omega)\geq\frac{1}{1+4\sqrt{2}A r}
  \]
\end{cor}

\begin{proof}
  Fix $0<\eta<\frac{1}{1+4\sqrt{2}A r}$.  Then $\frac{1}{2r^2}+\frac{\sqrt{2}A}{r}-\frac{1}{2r^2\eta}<-\frac{\sqrt{2}A}{r}$, so we may choose $N_s$ satisfying $\frac{1}{2r^2}+\frac{\sqrt{2}A}{r}-\frac{1}{2r^2\eta}<N_s<-\frac{\sqrt{2}A}{r}$.  Let $s=\frac{1}{1+2\sqrt{2}A r}$, and note that $s>\eta$ necessarily.  Fix $p\in b\Omega$.  Let $\ell_p$ denote the smallest eigenvalue of the Levi-form at $p$.  For $\tau\in T^{1,0}(b\Omega)$, we have a decomposition $\tau|_p=\tau_1+\tau_2$ where $\tau_1\in\mathcal{N}_p(\Omega)$ (which may be trivial) and $\tau_2$ is orthogonal to $\mathcal{N}_p(\Omega)$.  Hence, $i\ddbar\tilde\delta_\Omega(\tau,\bar\tau)\geq\ell_p|\tau_2|^2$.  At $p$,
  \begin{multline*}
    N_s\abs{\tau}^2+\frac{\sqrt{2}}{A r}\abs{\alpha_\Omega(\tau)}^2\leq\\
    \left(N_s+\frac{\sqrt{2}A}{r}\right)\abs{\tau_1}^2+N_s\abs{\tau_2}^2+\frac{2\sqrt{2}}{r}|\tau_1|\abs{\alpha_\Omega(\tau_2)}+\frac{\sqrt{2}}{A r}\abs{\alpha_\Omega(\tau_2)}^2.
  \end{multline*}
  For any $0<\epsilon<-\left(N_s+\frac{\sqrt{2}A}{r}\right)$, there exists $C_\epsilon>0$ such that
  \[
    N_s\abs{\tau}^2+\frac{\sqrt{2}}{A r}\abs{\alpha_\Omega(\tau)}^2\leq
    -\epsilon\abs{\tau_1}^2+C_\epsilon\abs{\tau_2}^2\leq-\epsilon\abs{\tau}^2+(C_\epsilon+\epsilon)\ell_p^{-1}i\ddbar\tilde\delta_\Omega(\tau,\bar\tau).
  \]
  Since $\beta$ is a positive semi-definite form, we have a neighborhood $U_p$ of $p$ such that
  \[
    N_s\abs{\tau}^2+\frac{\sqrt{2}}{A r}\abs{\alpha_\Omega(\tau)|_{\xi_\Omega(z)}}^2\leq
    (-\tilde\delta_\Omega(z))^{-1}i\ddbar\tilde\delta_\Omega(\tau,\bar\tau)|_{\xi_\Omega(z)}+2\beta_\Omega(\tau,\bar\tau)|_{\xi_\Omega(z)}
  \]
  for all $z\in U_p\cap\Omega$ and $\tau\in T^{1,0}_z$ satisfying $\partial\tilde\delta_\Omega(\tau)|_z=0$.  Since $\frac{s}{1-s}=\frac{1}{2\sqrt{2}A r}$ and $b\Omega$ is compact, there exists a neighborhood $U_s$ of $b\Omega$ such that \eqref{eq:varphi_hessian_estimate_hypothesis} is satisfied on $U_s\cap\Omega$ for $\varphi_s=0$.

  Now $\frac{1}{2r^2}\left(\frac{1}{\eta}-\frac{1}{s}\right)+N_s>0$, so we may choose $0<M_\eta<\frac{1}{2r^2}\left(1-\frac{\eta}{s}\right)+\eta N_s$ so that Proposition \ref{prop:C2_equivalence} $(2)$ implies the existence of $U_\eta$ and a $C^2$ defining function $\rho_\eta$ such that $-(-\rho_\eta)^\eta$ is strictly plurisubharmonic on $U_\eta\cap\Omega$.  Using Lemma \ref{lem:psh_extensions}, we may extend $\rho_\eta$ to all of $\Omega$, so $DF_{C^2}(\Omega)>\eta$ whenever $0<\eta<\frac{1}{1+4\sqrt{2}A r}$.
\end{proof}

It is also of interest to consider weight functions with self-bounded gradients, as defined by McNeal in \cite{McN02}.  Recall that a domain $\Omega\subset\mathbb{C}^n$ is said to satisfy Property $(\tilde P)$ if for every $B>0$ there exists a plurisubharmonic function $\phi\in C^2(\overline\Omega)$ such that $i\ddbar\phi\geq i\partial\phi\wedge\dbar\phi$ on $\Omega$ and $i\ddbar\phi\geq B\omega$ on $b\Omega$.
\begin{cor}
\label{cor:DF_P_tilde}
  Let $\Omega\subset\mathbb{C}^n$ be a bounded pseudoconvex domain with $C^2$ boundary.  Let $\mathcal{N}_p(\Omega)$ and $A$ be as in Corollary \ref{cor:DF_universal}.  Suppose that for some $B>0$ there exists a neighborhood $U_B$ of $b\Omega$ and a function $\phi_B\in C^2(U_B)$ such that if $\tau\in\mathcal{N}_p(\Omega)$ for some $p\in b\Omega$, then we have $i\ddbar\phi_B(\tau,\bar\tau)|_p\geq B|\tau|^2$ and $i\ddbar\phi_B(\tau,\bar\tau)|_p\geq\abs{\partial\phi_B(\tau)|_p}^2$.
  Then
  \[
    DF_{C^2}(\Omega)\geq\frac{B}{16A^2+B}.
  \]
  In particular, if $\Omega$ satisfies Property $(\tilde P)$, then $DF_{C^2}(\Omega)=1$.
\end{cor}

\begin{proof}
  Let $0<\eta<\frac{B}{16A^2+B}$.  Fix
  \[
    \eta<s<\frac{B}{16A^2+B}\text{ and }0<N_s<\left(\frac{1}{2}\sqrt{\frac{(1-s)B}{s}}-2A\right)^2.
  \]
  Set $t=\frac{1-s}{2s}-\sqrt{\frac{1-s}{B s}}2A>0$.  Let $\varphi_s=t\phi_B$.  Fix $p\in b\Omega$ and $\tau\in\mathcal{N}_p(\Omega)$.  Then at $p$,
  \[
    i\ddbar\varphi_s(\tau,\bar\tau)\geq \frac{1}{2t}\abs{\partial\varphi_s(\tau)}^2+\frac{t}{2}B|\tau|^2.
  \]
  Using \eqref{eq:basic_inequality} with $\epsilon=\frac{2ts}{1-s}$, we have
  \[
    i\ddbar\varphi_s(\tau,\bar\tau)\geq\frac{s}{1-s}\abs{\partial\varphi_s(\tau)-2\alpha_\Omega(\tau)}^2+\left(\frac{t}{2}B-\frac{2ts}{1-s-2ts}4A^2\right)|\tau|^2.
  \]
  Note that
  \[
    \frac{t}{2}B-\frac{2ts}{1-s-2ts}4A^2=\left(\frac{1}{2}\sqrt{\frac{(1-s)B}{s}}-2A\right)^2>N_s.
  \]

  Let $\ell_p$ denote the smallest eigenvalue of the Levi-form at $p$.  For $\tau\in T^{1,0}(b\Omega)$, we have a decomposition $\tau|_p=\tau_1+\tau_2$ where $\tau_1\in\mathcal{N}_p(\Omega)$ (which may be trivial) and $\tau_2$ is orthogonal to $\mathcal{N}_p(\Omega)$.  Hence, $i\ddbar\tilde\delta_\Omega(\tau,\bar\tau)\geq\ell_p|\tau_2|^2$.  At $p$,
  \begin{multline*}
    N_s\abs{\tau}^2+\frac{s}{1-s}\abs{\partial\varphi_s(\tau)-2\alpha_\Omega(\tau)}^2-i\ddbar\varphi_s(\tau,\bar\tau)\\
    \leq\left(N_s-\left(\frac{1}{2}\sqrt{\frac{(1-s)B}{s}}-2A\right)^2\right)\abs{\tau_1}^2+N_s\abs{\tau_2}^2\\
    +\frac{s}{1-s}(\abs{2\partial\varphi_s(\tau_1)-2\alpha_\Omega(\tau_1)}\abs{2\partial\varphi_s(\tau_2)-2\alpha_\Omega(\tau_2)}+\abs{2\partial\varphi_s(\tau_2)-2\alpha_\Omega(\tau_2)}^2)\\
    -2\re(i\ddbar\varphi_s(\tau_1,\bar\tau_2))-\abs{i\ddbar\varphi_s(\tau_2,\bar\tau_2)}^2.
  \end{multline*}
  Since first and second derivatives of $\varphi$ are uniformly bounded on $b\Omega$, there exists a sufficiently large constant $C>0$ such that
  \[
    N_s\abs{\tau}^2+\frac{s}{1-s}\abs{\partial\varphi_s(\tau)-2\alpha_\Omega(\tau)}^2-i\ddbar\varphi_s(\tau,\bar\tau)
    \leq C\abs{\tau_2}^2\leq C\ell_p^{-1}i\ddbar\tilde\delta_\Omega(\tau,\bar\tau).
  \]
  Since $\beta$ is a positive semi-definite form, we have a neighborhood $U_p$ of $p$ such that \eqref{eq:varphi_hessian_estimate_hypothesis} holds on $U_p\cap\Omega$.  Since $b\Omega$ is compact, we conclude that \eqref{eq:varphi_hessian_estimate_hypothesis} holds on $U_s\cap\Omega$ for some neighborhood $U_s$ of $b\Omega$.  Combining the conclusions of Proposition \ref{prop:C2_equivalence} $(2)$ with Lemma \ref{lem:psh_extensions}, we have a $C^2$ defining function $\rho_\eta$ such that $-(-\rho_\eta)^\eta$ is plurisubharmonic on $\Omega$.

  Since Property $(\tilde P)$ guarantees the existence of $\phi_B$ for every $B>0$, it follows that $DF_{C^2}(\Omega)=1$.
\end{proof}

\section{The Strong Diederich-Forn{\ae}ss Index on $C^k$ Domains}

On $C^k$ domains, with $k\geq 3$, we have a considerably simpler characterization of the strong Diederich-Forn{\ae}ss index.  As in Liu's work \cite{Liu17a}, this demonstrates that the strong Diederich-Forn{\ae}ss index on $C^3$ domains can be completely understood in terms of the existence of good weight functions on the boundary of the domain.  Liu's work has the advantage of working with arbitrary defining functions, while we require the use of the signed distance function.  However, since \eqref{eq:Weinstock} does not hold for an arbitrary defining function $\rho$, we must use third derivatives of $\rho$ to estimate the difference between the hessian of $\rho$ at $x$ and the hessian of $\rho$ at $\xi_\Omega(x)$, so we would lose our results on $C^2$ domains.  This is in close parallel to the comparison between Diederich and Forn{\ae}ss's original result on $C^2$ domains \cite{DiFo77b}, which use the distance function (or the pullback of this function to a Stein manifold embedded in $\mathbb{C}^n$), and Range's proof of the same result on $C^3$ domains \cite{Ran81}, which used an arbitrary defining function.

\begin{prop}
\label{prop:Ck_equivalence}
  Let $\Omega\subset\mathbb{C}^n$ be a bounded pseudoconvex domain with $C^k$ boundary, $3\leq k\leq \infty$, and let $r=\sup_\Omega|z|$.
  \begin{enumerate}
    \item Suppose that for some $0<\eta<1$ there exists a $C^2$ defining function $\rho_\eta$ for $\Omega$ with the property that $\lambda=-(-\rho_\eta)^\eta$ is plurisubharmonic on $U_\eta\cap\Omega$ for some neighborhood $U_\eta$ of $b\Omega$.  Then for every $0<s<\eta$ there exists a neighborhood $U_s\subset U_\eta$ of $b\Omega$ and a function $\varphi_s\in C^\infty(U_s)$ such that for every $\tau\in T^{1,0}(b\Omega)$, on $b\Omega$ we have
        \begin{equation}
        \label{eq:varphi_hessian_estimate_conclusion_smooth}
          \frac{s}{1-s}\abs{\partial\varphi_s(\tau)-2\alpha(\tau)}^2<i\ddbar\varphi_s(\tau,\bar\tau)+2\beta(\tau,\bar\tau).
        \end{equation}

    \item Suppose that for some $0<s<1$ there exists a neighborhood $U_s$ of $b\Omega$ and a function $\varphi_s\in C^k(U_s)$ such that for every $\tau\in T^{1,0}(b\Omega)$, on $b\Omega$ we have
        \begin{equation}
        \label{eq:varphi_hessian_estimate_hypothesis_smooth}
          \frac{s}{1-s}\abs{\partial\varphi_s(\tau)-2\alpha(\tau)}^2\leq i\ddbar\varphi_s(\tau,\bar\tau)
          +2\beta(\tau,\bar\tau).
        \end{equation}
        Then for every $0<\eta<s$, there exists a $C^k$ defining function $\rho_\eta$ for $\Omega$ with the property that $\lambda=-(-\rho_\eta)^\eta$ is strictly plurisubharmonic on $U_\eta\cap\Omega$ for some neighborhood $U_\eta$ of $b\Omega$.
  \end{enumerate}
\end{prop}

\begin{proof}
  We first assume that we are given $\rho_\eta$ and $U_\eta$.  For $0<s<\eta$, choose $s<\tilde s<\eta$ and $0<N_{\tilde s}<\frac{1}{2r^2}\left(\frac{1}{\tilde s}-\frac{1}{\eta}\right)$.  Let $\tilde U_{\tilde s}$ and $\tilde\varphi_{\tilde s}$ be given by Proposition \ref{prop:C2_equivalence}.  For $t>0$ such that $\{z\in\Omega:\delta_\Omega(z)=t\}\subset\tilde U_{\tilde s}$, let $U_{t,\tilde s}=\{z\in\tilde U_{\tilde s}:z-t\nabla\tilde\delta_\Omega(z)\in\tilde U_{\tilde s}\}$.  On $U_{t,\tilde s}$, define $\varphi_{t,\tilde s}(z)=\tilde\varphi_{\tilde s}(z-t\nabla\tilde\delta_\Omega(z))$.

  Now, for $z\in b\Omega$ we compute
  \begin{multline*}
    i\ddbar\varphi_{t,\tilde s}(z)\geq i\ddbar\tilde\varphi_{\tilde s}(z-t\nabla\tilde\delta_\Omega(z))\\-O\left(t\norm{\tilde\varphi_{\tilde s}}_{C^1(\tilde U_{\tilde s})}\norm{\tilde\delta_\Omega}_{C^3(U)}+t\norm{\nabla^2\tilde\varphi_{\tilde s}}_{L^\infty(\{z\in\Omega:\delta_\Omega(z)=t\})}\norm{\tilde\delta_\Omega}_{C^2(U)}^2\right).
  \end{multline*}
  By \eqref{eq:varphi_hessian_vanishing}, this error term is of order $o(1)$ as $t\rightarrow 0^+$.  Since $\tilde\delta_\Omega(z-t\nabla\tilde\delta_\Omega(z))=-t$ and $\xi_\Omega(z-t\nabla\tilde\delta_\Omega(z))=z$ for $z\in b\Omega$, we may use \eqref{eq:varphi_hessian_estimate_conclusion} for any $\tau\in T^{1,0}(b\Omega)$ to obtain
\begin{multline*}
  N_{\tilde s}\abs{\tau}^2+\frac{\tilde s}{1-\tilde s}\abs{\partial\tilde\varphi_{\tilde s}(\tau)|_{z-t\nabla\tilde\delta_\Omega(z)}-2\alpha_\Omega(\tau)}^2-o(1)|\tau|^2\leq\\ i\ddbar\varphi_{t,\tilde s}(\tau,\bar\tau)+t^{-1}i\ddbar\tilde\delta_\Omega(\tau,\bar\tau)+2\beta_\Omega(\tau,\bar\tau).
\end{multline*}
Clearly
\[
  \abs{\partial\tilde\varphi_{\tilde s}(z-t\nabla\tilde\delta_\Omega(z))-\partial\varphi_{t,\tilde s}(z)}\leq O\left(t\norm{\tilde\varphi_{\tilde s}}_{C^2(\tilde U_{\tilde s})}\norm{\tilde\delta_\Omega}_{C^1(U)}\right),
\]
so we may use \eqref{eq:basic_inequality} with $\epsilon=\frac{(1- \tilde s)s}{\tilde s(1-s)}$ to obtain
\begin{multline*}
  \abs{2\alpha_\Omega(\tau)-\partial\tilde\varphi_{\tilde s}(\tau)|_{z-t\nabla\tilde\delta_\Omega(z)}}^2\geq\frac{(1- \tilde s)s}{\tilde s(1-s)}\abs{2\alpha_\Omega(\tau)-\partial\varphi_{t,\tilde s}(\tau)|_{z}}^2\\
  -O\left(\frac{(1- \tilde s)s t^2}{\tilde s-s}\norm{\tilde\varphi_{\tilde s}}_{C^2(\tilde U_{\tilde s})}^2\norm{\tilde\delta_\Omega}_{C^1(U)}^2|\tau|^2\right).
\end{multline*}
Hence,
\begin{multline*}
  N_{\tilde s}\abs{\tau}^2+\frac{s}{1-s}\abs{\partial\varphi_{t,\tilde s}(\tau)-2\alpha_\Omega(\tau)}^2-o(1)|\tau|^2\leq\\ i\ddbar\varphi_{t,\tilde s}(\tau,\bar\tau)+t^{-1}i\ddbar\tilde\delta_\Omega(\tau,\bar\tau)+2\beta_\Omega(\tau,\bar\tau).
\end{multline*}
on $b\Omega$ for all $\tau\in T^{1,0}(b\Omega)$.  If we choose $t$ small enough so that the error terms are bounded by $\frac{1}{2}N_{\tilde s}|\tau|^2$, then we have
\[
  \frac{1}{2}N_{\tilde s}\abs{\tau}^2+\frac{s}{1-s}\abs{\partial\varphi_{t,\tilde s}(\tau)-2\alpha_\Omega(\tau)}^2\leq i\ddbar\varphi_{t,\tilde s}(\tau,\bar\tau)+t^{-1}i\ddbar\tilde\delta_\Omega(\tau,\bar\tau)+2\beta_\Omega(\tau,\bar\tau).
\]
on $b\Omega$ for all $\tau\in T^{1,0}(b\Omega)$.  If we regularize $\varphi_{t,\tilde s}+t^{-1}\tilde\delta_\Omega$ by convolution in some neighborhood of $b\Omega$, we obtain $\varphi_s$ satisfying \eqref{eq:varphi_hessian_estimate_conclusion_smooth}.

Conversely, suppose that we are given $\varphi_s$ and $U_s$.  Since $\Omega$ is pseudoconvex, $i\ddbar\tilde\delta_\Omega(\tau,\bar\tau)\geq 0$ for all $\tau\in T^{1,0}(b\Omega)$.  For any $0<\eta<s$ and $-\frac{1}{2r^2}\left(\frac{1}{\eta}-\frac{1}{s}\right)<N_s<0$, \eqref{eq:varphi_hessian_estimate_hypothesis_smooth} will imply \eqref{eq:varphi_hessian_estimate_hypothesis} in some neighborhood of $b\Omega$.  Hence we obtain $\rho_\eta$ and $U_\eta$ by Proposition \ref{prop:C2_equivalence}.  Since $\rho_\eta$ is a $C^k$ multiple of $\tilde\delta_\Omega$, $\rho_\eta$ will be at least as smooth as $\tilde\delta_\Omega$, which is $C^k$ in a neighborhood of $b\Omega$ by \cite{KrPa81}.
\end{proof}

\begin{cor}
\label{cor:strong_DF_Index}
  Let $\Omega\subset\mathbb{C}^n$ be a bounded pseudoconvex domain with $C^k$ boundary for $k\geq 2$.  Then $DF_{C^k}(\Omega)=DF_s(\Omega)$.
\end{cor}

\begin{proof}
  This is trivial for $k=2$, so we assume $k\geq 3$.  Clearly $DF_{C^2}(\Omega)\geq DF_{C^k}(\Omega)$.  Suppose there exists $DF_{C^2}(\Omega)>\eta>DF_{C^k}(\Omega)$.  Then there exists a $C^2$ defining function $\rho_\eta$ for $\Omega$ such that $-(-\rho_\eta)^\eta$ is plurisubharmonic on $\Omega$.  Choose $DF_{C^k}(\Omega)<s<\eta$.  Let $U_s$ and $\varphi_s$ be given by Proposition \ref{prop:Ck_equivalence} $(1)$.  Choose $DF_{C^k}(\Omega)<\tilde\eta<s$.  Let $U_{\tilde\eta}$ and $\rho_{\tilde\eta}$ be given by Proposition \ref{prop:Ck_equivalence} $(2)$.  If we use Lemma \ref{lem:psh_extensions} to extend $-(-\rho_{\tilde\eta})^{\tilde\eta}$ to all of $\Omega$, we have a contradiction.
\end{proof}

\bibliographystyle{amsplain}
\bibliography{harrington}
\end{document}